\documentclass{amsart}
\usepackage[utf8]{inputenc}
\usepackage{amsmath}
\usepackage{amsfonts}
\usepackage{bbm}
\usepackage{amsthm}
\usepackage{graphicx}
\usepackage{tikz}
\usepackage{caption}
\usepackage{subfig}
\usepackage[autostyle]{csquotes}
\usepackage{xspace}
\usepackage[
text={440pt,575pt},
headheight=9pt,
centering
]{geometry}
\usepackage{xcolor}
\usepackage{hyperref}
\usepackage{url, hypcap}
\definecolor{darkblue}{RGB}{0,0,160}
\hypersetup{
	colorlinks,
	citecolor=darkblue,
	filecolor=black,
	linkcolor=darkblue,
	urlcolor=darkblue
}
\usepackage{mathtools}

\newtheorem{theorem}{Theorem}[section]
\newtheorem{lemma}[theorem]{Lemma}
\newtheorem{proposition}[theorem]{Proposition}

\theoremstyle{definition}
\newtheorem{example}[theorem]{Example}
\newtheorem{remark}[theorem]{Remark}
\newtheorem{definition}[theorem]{Definition}
\newtheorem{assumption}[theorem]{Assumption}
\newtheorem{observation}[theorem]{Observation}

\numberwithin{equation}{section}

\newcommand{\EE}{\mathbb{E}}
\newcommand{\PP}{\mathbb{P}}
\newcommand{\VV}{\mathbb{V}}

\newcommand{\Mathematica}{\textsc{Mathematica}\xspace}
\newcommand{\sage}{\textsc{Sage}\xspace}

\newcommand{\on}[1]{\operatorname{#1}}
\newcommand{\des}{\operatorname{des}}
\newcommand{\rk}{\operatorname{rk}}
\newcommand{\cox}[1]{\mathtt{#1}}

\allowdisplaybreaks

\begin{document}
\title{A central limit theorem for the two-sided descent statistic on Coxeter groups}
\author[B.~Br\"uck]{Benjamin Br\"uck}
\address[]{Department of Mathematics, ETH Z\"urich \\ Rämistrasse 101, 8092 Z\"urich, Switzerland}
\email{benjamin.brueck@math.ethz.ch}
\urladdr{\url{https://people.math.ethz.ch/~bbrueck/}}
\author[F.~R\"ottger]{Frank R\"ottger}
\address[]{Research Center for Statistics, University of Geneva, Boulevard du Pont d’Arve 40, 1205 Geneva, Switzerland}
\email{frank.roettger@unige.ch}
\urladdr{\url{https://sites.google.com/view/roettger/}}

\begin{abstract}
 We study the asymptotic behaviour of the statistic $(\on{des}+\on{ides})_W$ which assigns to an element $w$ of a finite Coxeter group $W$ the number of descents of $w$ plus the number of descents of $w^{-1}$. Our main result is a central limit theorem for the probability distributions associated to this statistic. This answers a question of Kahle--Stump and builds upon work of Chatterjee--Diaconis, \"Ozdemir and R\"ottger.
 
 \smallskip
 \noindent \textbf{Keywords.} Probabilistic combinatorics, Coxeter statistics, Central limit theorem, Descent statistic
\end{abstract}

\maketitle

\section{Introduction}
Statistical and probabilistic methods in the investigation of combinatorial and algebraic objects are powerful tools and reveal deeply rooted connections between those fields. Of greatest significance in probabilistic asymptotics is the central limit theorem (CLT), that is the convergence in distribution of a sequence of random variables, normalised by its mean and its standard deviation, towards the standard Gaussian.
This paper's main result is an equivalent formulation of the central limit theorem for a sequence of random variables that arises from a statistic on sequences of finite Coxeter groups.

In the symmetric group $ \on{Sym}(n) $, which is the Coxeter group of type $\cox{A}_{n-1}$, the descent statistic is defined as follows: Write the elements of $\on{Sym}(n)$ in one-line notation. Then the number of descents  $\on{des}(\pi)$ of an element $\pi\in \on{Sym}(n)$ is given by the number of positions where an entry is larger than its successor. This concept generalises to arbitrary finite Coxeter groups, the necessary definitions are presented in Section~\ref{s:Descents}.

Fixing such a Coxeter group $W$, choosing an element of $W$ uniformly at random and evaluating the descent statistic gives rise to a random variable $D_W$.
Kahle and Stump recently showed that for sequences $(W_n)_n$ of finite Coxeter groups of growing rank\footnote{The rank of a Coxeter group $W$ is the size of a particular generating set of $W$ (the set of ``simple reflections''). It can be seen as a measure of size or complexity of the group. The rank of the symmetric group is $\rk(\on{Sym}(n)) = n-1$.}, the sequence $D_{W_n}$ satisfies the CLT if and only if its variance tends to infinity, see \cite{KahleStump2018}.
They asked \cite[Problem 6.10]{KahleStump2018} whether for the random variable $T_W$ associated to the statistic $t(w)\coloneqq \on{des}(w)+\on{des}(w^{-1})$, a similar statement holds true.
The statistic $t$ was studied by Chatterjee--Diaconis \cite{chatterjee2017} who were motivated by defining a metric using descents; it also has a geometric interpretation in terms of a two-sided analogue of the Coxeter complex first introduced by Hultman  \cite{Hultman2007} and also studied by Petersen \cite{Pet:twosidedanalogue}, for details see Appendix~\ref{appendixGeometry}.
Our main result is a positive answer to the question of Kahle--Stump under an additional hypothesis on the sequence of Coxeter groups. This hypothesis does not seem to be very restrictive, see the comments below.	
\begin{theorem}\label{t:main_result_intro}
	Let $(W_n)_n$ be a well-behaved sequence of finite Coxeter groups such that $\rk(W_n)\to \infty$ and let $T_n$ be the random variable associated to the statistic $t$ on $W_n$. Then the following are equivalent:
	\begin{enumerate}
		\item $(T_n)_n$ satisfies the CLT;
		\item \label{i:Variancetoinftyintro} $\VV(T_n)\rightarrow \infty$.
	\end{enumerate}
\end{theorem}
Item~\ref{i:Variancetoinftyintro} can equivalently be defined in terms of the irreducible components of $W_n$ (see Theorem~\ref{t:mainresult}). It is in particular satisfied if the maximal size of a dihedral parabolic subgroup in $W_n$ does not grow too fast, e.g.~if it is bounded.
We give the definition of ``well-behaved'' and sufficient conditions in Section~\ref{s:mainresult} but would like to remark that we were not able to construct a sequence of Coxeter groups that does not have this property. In particular, sequences $(W_n)_n$ that satisfy Item~\ref{i:Variancetoinftyintro} are well-behaved if: the number or rank of irreducible factors occurring in any $W_n$ are bounded; or there are no irreducible factors of dihedral type and only boundedly many irreducible components of $W_n$ have rank not in $o(\rk(W_n))$.
An example of a well-behaved sequence with an unbounded number of irreducible components that have rank not in $o(\rk(W_n))$ is given in Example~\ref{ex:CoxetergroupsandCLT} (i).

We would like to point out that since the first publication of this article, Valentin Féray \cite{Feray2020} has shown how to remove the well-behaved condition from Theorem~\ref{t:main_result_intro}.
This is achieved by applying an inequality of Mallows \cite{Mallows72} to our proof, improving on Lemma~\ref{l:infinitesum}. This allows to control better the convergence of the irreducible components that do not have rank in $o(\rk(W_n))$, such that the well-behaved condition is not required.

Special cases of Theorem~\ref{t:main_result_intro} were known before: For the case where $W_n=\on{Sym}(n+1)$, the irreducible Coxeter group of type $ \cox{A}_n $, the result is due to Vatutin \cite{Vatutin} and was later, with different methods, reproven by Chatterjee--Diaconis \cite{chatterjee2017} and \"Ozdemir \cite{oezdemir2019}. Following the approach of Chatterjee and Diaconis, R\"ottger \cite{Roettger2019} generalised this to the cases where $W_n$ is an irreducible Coxeter group of type $ \cox{B}_n $ or $ \cox{D}_n $. Technical difficulties of these proofs lie in the dependencies between $ \on{des}(w) $ and $ \on{des}(w^{-1}) $, which require probabilistic methods as for example interaction graphs, see \cite{chatterjee2008}, to establish the CLT.

In order to extend these results to arbitrary products of irreducible Coxeter groups, we take an approach similar to the one used by Kahle--Stump \cite{KahleStump2018} for the descent statistic; this in particular involves an application of Lindeberg's theorem for triangular arrays. There is however a major difference between their approach and ours: The generating function of the descent statistic is given by the Eulerian polynomial which factors over the reals and has only negative roots, see \cite{Bre:qEulerianpolynomials} and \cite{SV:sEulerianpolynomials}. Kahle and Stump heavily used this in order to deduce their result. In contrast to that, the generating function of the statistic $t$ is the two-sided Eulerian polynomial as studied e.g. in \cite{CRS:Permutationssequencesrepetitions}, \cite{Pet:TwosidedEulerian} and \cite{Vis:Someremarksjoint}. It does not have a such a nice factorisation, even in the setting of symmetric groups. In order to resolve the additional difficulties arising from this, we are led to compute higher moments of the random variables $T_W$. For this, we use and generalise the work of Özdemir \cite{oezdemir2019}.

\subsection*{Structure of article}
The structure of the paper is as follows:
Section~\ref{s:Descents} introduces some basic notations, finite Coxeter groups and the descent statistic.
Section~\ref{s:fourthmoments} explains how to derive recursively higher moments of the descent statistic and the statistic $t$. This is done using conditional expectations and a recursion solver software.
In Section~\ref{s:rank=O(n)}, we give sufficient conditions for establishing the CLT for weighted sums of sequences of random variables which all individually satisfy the CLT.
These enable us in Section~\ref{s:Lindeberg} to apply the Lindeberg Theorem and obtain the asymptotic normality of $T_{W_n}$ for sequences of Coxeter groups $W_n$ whose irreducible components satisfy a certain maximum condition.
Combining these results, Section~\ref{s:mainresult} delivers our main theorem.  
In the appendix we present a discussion of a geometric perspective on the statistic $t$ in the context of the two-sided analogue of the Coxeter complex defined in \cite{Hultman2007}, as well as a table of moments of the statistics $\on{des}$ and $t$ for Coxeter groups of type $\cox{A}$ and $\cox{B}$.

\subsection*{Acknowledgements}
We express special gratitude towards Norbert Gaffke for his help with the technical difficulties in Section~\ref{s:rank=O(n)}. We would like to thank Thomas Kahle and Hauke Seidel for helpful remarks on a first version of this text.
Part of this work was established during a visit of the first-named author at OVGU Magdeburg. He would like to thank the group there for their hospitality. 
Benjamin Br\"uck was supported by the grant BU 1224/2-1 within the Priority Programme 2026 ``Geometry at infinity'' of the Deutsche Forschungsgemeinschaft (DFG). Frank R\"ottger acknowledges support by the Deutsche Forschungsgemeinschaft (DFG)
under grant 314838170, GRK 2297 MathCoRe.

\section{Preliminaries} \label{s:Descents}

\subsection{Central limit theorems and o-notation}
Let $ (X_n)_n $ be a sequence of random variables with distribution functions $ (F_n)_n $.
We say that $ (X_n)_n $ converges in distribution to a random variable $ X $ with distribution function $ F $ (denoted as $ X_n\stackrel{D}{\rightarrow}X $), if for every $ x $ where $ F $ is continuous, we have $ \lim_{n\to \infty}F_n(x)=F(x) $.

We say that a sequence of integrable random variables $ (X_n)_n $ with finite variance \emph{satisfies the central limit theorem (CLT)}, if it holds that 
\[\frac{X_n-\EE(X_n)}{\sqrt{\VV(X_n)}}\stackrel{D}{\rightarrow}N(0,1),\]
which means that $ (X_n)_n $, normalised by its mean and its standard deviation, converges in distribution towards the standard Gaussian. 

The following will become useful for establishing CLTs later on:

\begin{lemma}
	\label{l:subsubsequence}
	Let $(X_n)_n$ be a sequence of integrable random variables with finite variance. Then $(X_n)_n$ satisfies the CLT if and only if every subsequence of $(X_n)_n$ has a subsequence which satisfies the CLT.
\end{lemma}
\begin{proof}
	This follows from the following elementary fact: Let $(a_n)_n$ be a sequence in a topological space $A$ and let $a\in A$. If every subsequence of $(a_n)_n$ has a subsequence which converges to $a$, then $(a_n)_n$ converges to $a$.
	Apply this to the sequence of distribution functions.
\end{proof}

In this paper, we use little-$ o $ and big-$ O $ notation. The definitions vary in the literature, we use the following conventions: Let $f$ and $g$ be maps from $\mathbb{N}_+$ or $\mathbb{R}_{\geq 0}$ to $\mathbb{R}_{\geq 0}$. We say that $ f(n)=o(g(n)) $, if it holds that $ \lim_{n\rightarrow \infty} \frac{f(n)}{g(n)}=0 $. Furthermore, we write $ f(n)=O(g(n)) $, if there is a constant $ C > 0 $ and $N \in \mathbb{N}$ such that for all  $n \geq N$, one has $ f(n)\leq C g(n)$. We say that $ f(n) $ is \emph{of order} $ g(n) $, if $ \lim_{n\rightarrow \infty} \frac{f(n)}{g(n)}=c $ where $ c $ is a positive constant.

\subsection{Coxeter groups}
We start with recalling some background about Coxeter groups. For further details, we refer the reader to \cite{BB:CombinatoricsCoxetergroups}.

Let $S$ be a set. A matrix $m:S\times S \to \mathbb{N}\cup \{\infty\}$ is called a \emph{Coxeter matrix}, if for all $(s,s')\in S\times S$, the following holds true:
\begin{align*}
	m(s,s')=m(s',s)\geq 1,\\
	m(s,s')=1 \Leftrightarrow s=s'.
\end{align*}
A group $W$ is called a \emph{Coxeter group}, if there is a set $S\subseteq W$ and a Coxeter matrix $m:S\times S \to \mathbb{N}\cup \{\infty\}$ such that a presentation of $W$ is given by
\begin{align*}
	W=\left\langle S \,\middle|\, (ss')^{m(s,s')}=1 \text{ for all } (s,s')\in S\times S\right\rangle.
\end{align*}
In this setting, the pair $(W,S)$ is called a \emph{Coxeter system} and $S$ the set of \emph{simple reflections}. The size of $S$ is called the \emph{rank of $(W,S)$}, abbreviated by $\rk(W)$. In what follows, when we talk about a Coxeter group $W$, we tacitly assume that it comes with a fixed set generating set $S$ which make $(W,S)$ a Coxeter system. Also, if we write $W$ as a product of Coxeter groups $W= W_1 \times W_2 \times \cdots \times W_n$, we assume that $S= S_1 \cup S_2 \cup \ldots \cup S_n$, where $S_i$ is the set of simple reflections of $W_i$.

A Coxeter group $W$ is called \emph{irreducible} if it cannot be written as a non-trivial product of Coxeter groups $W=W_1\times W_2$.
By the classification of finite reflection groups, every \emph{finite} irreducible Coxeter group falls into one of the four infinite families $\cox{A}_n$, $\cox{B}_n$, $\cox{D}_n$, $\cox{I}_2(m)$ or is isomorphic to one of seven finite reflection groups of exceptional type. For combinatorial descriptions of the groups of type $\cox{A}_n$, $\cox{B}_n$, $\cox{D}_n$, see \cite[Chapter 8]{BB:CombinatoricsCoxetergroups}.
A Coxeter group $W$ is said to be a \emph{dihedral group} or \emph{of dihedral type} if $\rk(W)=2$; if $W$ is irreducible, this is equivalent to saying that it is of type $\cox{I}_2(m)$ for some $m\geq 3$. 
Any finite Coxeter group $W$ can be written as a product
\begin{align*}
	W= W_1 \times W_2 \times \cdots \times W_k,
\end{align*}
where each $W_i$ is an irreducible Coxeter group. This decomposition is unique up to permutation of the factors and we call the $W_i$ the \emph{irreducible components of $W$}.

\begin{example}
	\label{ex_symmetric_group}
	Let $W = \on{Sym}(n)$ be the symmetric group on an $n$ element set and let $S$ be the set of pairwise adjacent transpositions $\{(i,i+1)|1 \leq i \leq n-1 \}$. Then $(W,S)$ is a Coxeter system of rank $|S| = n-1$. This gives $\on{Sym}(n)$ the structure of the irreducible Coxeter group of type $\cox{A}_{n-1}$.
\end{example}

\subsection{Coxeter statistics}
In this subsection, we fix a finite Coxeter group $W$ with a set $S$ of simple reflections.
Given an element $w\in W$, the \emph{descent set of $w$} is defined by
\begin{align*}
	\on{Des}(w)\coloneqq \left\lbrace s\in S \, \middle| \, l_S(ws)<l_S(w) \right\rbrace,
\end{align*}
where $l_S(w)$ is the \emph{length of $w$ with respect to $S$}, i.e. the smallest number $n$ such that $w=s_1 s_2 \cdots s_n$, where $s_i\in S$ for all $i$.
The \emph{number of descents} gives rise to a statistic $\on{des}:W  \to \mathbb{N}$ on $W$ defined by $\des(w)\coloneqq |\on{Des}(w)|$.
Choosing an element of $W$ uniformly at random and evaluating this statistic yields a random variable $D$ on $\mathbb{N}$.

\begin{example}
	Similar to Example~\ref{ex_symmetric_group}, let $W$ be the symmetric group $\on{Sym}(3)$ and $S$ its set of pairwise adjacent transpositions $S = \{s_1 = (12), s_2 = (23)\}$. Let $w\in \on{Sym}(3)$ be the 3-cycle $(123)$. Then $w$ can be written in terms of the simple reflections as $w = s_2 s_1$. We have $l_S(w) = 2$, $\on{Des}(w) = \{ s_1 \}$ and $\des(w) = |\on{Des}(w)| = 1$.
\end{example}

The aim of this article is to study the behaviour of the statistic $t$ defined by
\begin{align*}
	t: W & \to \mathbb{N}\\ 
	w &\mapsto  \des(w)+ \des(w^{-1}).
\end{align*}
Just like $\des$, when we choose an element of $ W $ uniformly, the statistic $ t $ gives rise to a random variable on $\mathbb{N}$ which we denote by $T$. 

We also write $\des_W$, $D_W$, $t_W$ or $T_W$ if we want to emphasise the ambient Coxeter group corresponding to these statistics and random variables.

\begin{lemma}
	\label{l:sumofindependent}
	Assume that $W$ decomposes as a product $W_1\times W_2$ of Coxeter groups $W_1$ and $W_2$. Then $T_W$ can be written as a sum of  independent random variables $T_W= T_{W_1}+ T_{W_2}$.
\end{lemma}
\begin{proof}
	Let $S_1$ and $S_2$ be the set of simple reflections of $W_1$ and $W_2$, respectively. By assumption, we have $S=S_1\cup S_2$.
	Every $w\in W$ can be uniquely written as a product $w=w_1 w_2= w_2 w_1$, where $w_i\in W_i$ and one has $l_S(w)=l_{S_1}(w_1)+l_{S_2}(w_2)$. Consequently, $\des_W(w)=\des_{W_1}(w_1)+\des_{W_2}(w_2)$ and $t_W(w)=t_{W_1}(w_1)+t_{W_2}(w_2)$. The claim now follows because choosing an element of $W$ uniformly at random is equivalent to choosing uniformly at random $w_1$ from $W_1$ and independently $w_2$ from $W_2$.
\end{proof}

\begin{theorem}
	\label{t:varianceofT}
	Let $W$ be a finite Coxeter group and $T$ as above.
	\begin{enumerate}
		\item $\EE(T)=\rk(W)$.
		\item If $W$ is a product of dihedral groups,  $W=\prod_{i=1}^{k} \cox{I}_2(m_i)$, then $\VV(T)=\sum_{i=1}^{k} \frac{4}{m_i}$.
		\item If $W_n$ is a sequence of finite Coxeter groups such that for all $n$, every irreducible component of $W_n$ is of non-dihedral type, then $\VV(T_{W_n})$ is of order $\rk(W_n)$.
	\end{enumerate}
\end{theorem}
\begin{proof}
	Kahle--Stump computed the variance of $T$ for all types of finite irreducible Coxeter groups in \cite[Corollary 5.2]{KahleStump2018}. Using Lemma~\ref{l:sumofindependent} and additivity of the variance, the result follows immediately.
\end{proof}

\section{Fourth moments of T} \label{s:fourthmoments}
As defined in Section~\ref{s:Descents}, let $D_W$ be the random variable associated to the statistic $\des_{W}$ and let $T_W$ be the random variable associated to the statistic $t_{W}$ for a finite Coxeter group $W$.
The aim of this section is to prove the following theorem:

\begin{theorem}\label{t:fourthmoments}
	Let $W$ be an irreducible Coxeter group of type $ \cox{A}_n,\cox{B}_n $ or $ \cox{D}_n $. Then the fourth central moment $\EE((T_W-\EE(T_W))^4)$ of $ T_W $ is of order $ n^2 $. 
\end{theorem}

In order to show this, we follow and extend the ideas of Özdemir.
In \cite{oezdemir2019}, he formulated the conditional laws
\begin{align}\label{eq:recursionA}
	\EE(D_{\cox{A}_{n+1}}|D_{\cox{A}_n})=\begin{cases}
		D_{\cox{A}_n} & \text{with probability} ~~ \frac{D_{\cox{A}_n}+1}{n+1},\\
		D_{\cox{A}_n} +1 & \text{with probability} ~~ \frac{n-D_{\cox{A}_n}}{n+1}\\
	\end{cases}
\end{align} 
and  
\begin{align}\label{eq:recursionB}
	\EE(D_{\cox{B}_{n+1}}|D_{\cox{B}_{n}})=\begin{cases}
		D_{\cox{B}_{n}} & \text{with probability} ~~ \frac{2D_{\cox{B}_{n}}+1}{2n+2},\\
		D_{\cox{B}_{n}}+1 & \text{with probability} ~~ \frac{2n-2D_{\cox{B}_{n}}+1}{2n+2}.\\
	\end{cases}
\end{align} 
Here $ \EE(D_{W_{n+1}}|D_{W_{n}}) $ denotes the conditional expected value where $D_{W_{n+1}}$ is generated from $D_{W_{n}}$. For $ \cox{A}_n $, this is done in the one-line notation by inserting $ n+2 $ in a random position in the permutation of length $ n+1 $. For $ \cox{B}_n $, we insert $ n+1 $ multiplied with a binary random variable that assigns equal probability to $ \{\pm 1\} $ in a signed permutation of length $ n $. See \cite{oezdemir2019} for a detailed overview. 
Özdemir used these formulas to compute higher moments of $D_{\cox{A}_n}$ and $D_{\cox{B}_n}$. 
An important tool for his computations is the smoothing theorem (also known as the the law of total expectation) which can be stated as follows:
\begin{theorem}[{Smoothing Theorem; cf. \cite[Theorem 34.4]{billingsley1995probability}}]
	Let $ X $ and $ Y $ be integrable random variables defined on the same probability space. Then, it holds that
	\[\EE(\EE(X|Y))=\EE(X).\]
\end{theorem}

Our approach for proving Theorem~\ref{t:fourthmoments} is to compute inductively higher moments of $T_W$ and $D_W$ for the different families of Coxeter groups separately.
We start in Section~\ref{s:fourthmomentD} by computing the fourth central moment of $D_W$ in the case where $W$ is irreducible and of type $\cox{A}$ or $\cox{B}$. These computations serve as an illustration of the methods we use and the results will be needed for our inductive method of computing the fourth central moments of $T_W$ later on. Building on this, we prove Theorem~\ref{t:fourthmoments} for $W$ of type $\cox{A}$ and $\cox{B}$ in Section~\ref{s:fourthmomentTCoxA} and Section~\ref{s:fourthmomentTCoxB}, respectively.
We finish the proof in Section~\ref{s:prooffourthmoments}.

\subsection{Fourth moment of D}
\label{s:fourthmomentD}
Özdemir showed that the fourth central moment of the random variable $D_{\cox{A}_n }$ is of order $n^2$ \cite[p. 3]{oezdemir2019}. Using the \textbf{RSolve} function of \Mathematica, we are able to give an explicit formula for this moment:

\begin{lemma}\label{l:fmA}
	Let $ D_n $ be the random variable associated to the statistic $\des$ on the Coxeter group $ \cox{A}_n $, $ n\ge 3 $. Then we have:
	\begin{align*}
		\EE((D_n-\EE(D_n))^4)=\frac{1}{240} (n+2) (5 n+8).
	\end{align*}
\end{lemma}
\begin{proof}
	From (\ref{eq:recursionA}), we derive the recursion formula
	\begin{align}\label{eq:recursionA2}
		\EE((D_{n+1}-\EE(D_{n+1}))^4|D_n)=\frac{(n-2) (D_{n}-\EE(D_{n}))^4}{n+2}+\frac{(3 n+4) (D_{n}-\EE(D_{n}))^2}{2 (n+2)}+\frac{1}{16}.
	\end{align} 
	By applying $ \EE $ on both sides of \eqref{eq:recursionA2}, the smoothing theorem leads to
	\begin{align*}
		\EE((D_{n+1}-\EE(D_{n+1}))^4)=\frac{( n-2)\EE( (D_{n}-\EE(D_n))^4)}{ n+2}+\frac{(3 n+4) \text{Var}(D_n)}{2 (n+2)}+\frac{1}{16}
	\end{align*}
	and with the formula for the variance found for example in \cite[Corollary 5.2]{KahleStump2018}, we obtain a recursive formula for $ a[n]=\EE((D_{n}-\EE(D_n))^4) $:
	\begin{align*}
		a[n+1]=\frac{(6n+11)}{48}+\frac{( n-2) a[n]}{ n+2},
	\end{align*}
	which was solved by computing the value $ a[3]=\frac{23}{48} $ with \sage and using the \textbf{RSolve} function of \Mathematica.
\end{proof}

Using the same method and \eqref{eq:recursionB}, we can compute the same moment in type $\cox{B}$:

\begin{lemma}\label{l:fmB}
	Let $ D_n $ be the random variable associated to the statistic $\des$ on the Coxeter group $ \cox{B}_n $, $ n\ge 4 $. Then we have:
	\begin{align}\label{eq:fourthmomentdesB}
		\EE((D_n-\EE(D_n))^4)=\frac{1}{240} (n+1) (5 n+3).
	\end{align}
\end{lemma}
\begin{proof}
	From (\ref{eq:recursionB}), we derive the recursion formula
	\begin{align*}
		\EE((D_{n+1}-\EE(D_{n+1}))^4|D_n)=\frac{( n-3) (D_{n}-\EE(D_n))^4}{ n+1}+\frac{(3 n+1) (D_{n}-\EE(D_n))^2}{2 (n+1)}+\frac{1}{16}.
	\end{align*} 
	This is the same recursion formula as for type $ A_{n-1} $ in \eqref{eq:recursionA2}, so we obtain a recursive formula for $ a[n] =\EE((D_{n}-\EE(D_n))^4) $:
	\begin{align*}
		a[n+1]=\frac{(6n+5)}{48}+\frac{( n-3) a[n]}{ n+1},
	\end{align*}
	which was also solved by computing the starting value $ a[4]=\frac{23}{48} $ with \sage and using the \textbf{RSolve} function of \Mathematica.
\end{proof}

\subsection{Moments of T for type \texorpdfstring{$\cox{A}_n$}{A}}
\label{s:fourthmomentTCoxA}
Throughout this subsection, let $T_n=T_{\cox{A}_n}$, $D_n=D_{\cox{A}_n}$ and let $D_n'$ be the random variable associated to the statistic
\begin{align*}
	\cox{A}_n &\to \mathbb{N}\\
	w &\mapsto \des(w^{-1}).
\end{align*}
Clearly, we have $T_n=D_n+D'_n$, but $D_n$ and $D_n'$ are not independent.
In order to compute the fourth central moment of $T_n$, we want to determine inductively mixed moments of the form $\EE(D_{n}^k {D'_{n}}^l)$. To compute these moments recursively, we use the following two-dimensional conditional law for $ (D_{n+1},D'_{n+1})$ given $(D_{n},D'_{n}) $ introduced by Özdemir:

\begin{lemma}[see {\cite[p. 18]{oezdemir2019}}]\label{l:submartingaleA}
	In type $\cox{A}_n$, the conditional law of $ (D_{n+1},D'_{n+1})$ given $(D_{n},D'_{n}) $ is 
	 \begin{align*}
	 	\EE((D_{n+1},D'_{n+1})|(D_{n},D'_{n}))&=\begin{cases}
	 		(D_{n},D'_{n}) & \text{with prob.}~ P_1=\frac{(D_{n}+1)(D'_{n}+1)+n+1}{(n+2)^2},\\
	 		(D_{n}+1,D'_{n}) & \text{with prob.}~P_2= \frac{(n+1-D_{n})(D'_{n}+1)-n-1}{(n+2)^2},\\
	 		(D_{n},D'_{n}+1) & \text{with prob.}~ P_3=\frac{(D_{n}+1)(n+1-D'_{n})-n-1}{(n+2)^2},\\
	 		(D_{n}+1,D'_{n}+1) & \text{with prob.}~P_4=\frac{(n+1-D_{n})(n+1-D'_{n})+n+1}{(n+2)^2}.\\
	\end{cases}  
	\end{align*}
\end{lemma}
We remark that in comparison to this, there is a shift of indices in \cite[p. 18]{oezdemir2019} as there, $D_n$ corresponds to the descent statistic on $\on{Sym}(n)=\cox{A}_{n-1}$. Özdemir used this in order to compute the asymptotics of $\EE(\,(D_{n}-\EE(D_n))^2 (D'_{n}-\EE(D'_n))^2\,)$, see \cite[Lemma 5.1]{oezdemir2019}. We obtain his results and generalisations of it in the proof of the following proposition.

\begin{proposition}\label{p:fourthmomentA} 
	In type $\cox{A}_n$, $ n\geq 3 $, the fourth central moment of $T_n$ is given by
	\begin{align*}
		\EE((T_n-\EE(T_n))^4)=\frac{1}{60} \left(5 n^2+79 n+258\right)-\frac{5 n+2}{n(n+1)}.
	\end{align*}
\end{proposition}

\begin{proof}
	Define $U_n\coloneqq D_n-\EE(D_n)=D_n-n$ and $U'_n\coloneqq D'_n-\EE(D'_n)$. Our goal is to compute
	\begin{align*}
		\EE((T_n-\EE(T_n))^4)=\EE((U_n+U'_n)^4).
	\end{align*}
	Multiplying out the right hand side of this equation and using linearity of the expected value, we see that it suffices to compute $\EE(U_{n}^k {U'}_{n}^l)$ for all $0\leq k,l \leq 4$ with $k+l=4$.
	
	Using the smoothing theorem and Lemma~\ref{l:submartingaleA}, we derive the following recursion formula for fixed $k$ and $l$: 
	\begin{multline*}
		\EE\left(U_{n+1}^k ({U'_{n+1}})^l\right)=\EE\big(\left(U_{n}-\frac{1}{2}\right)^k \left(U'_{n}-\frac{1}{2}\right)^l P_1+\left(U_{n}+\frac{1}{2}\right)^k \left(U'_{n}-\frac{1}{2}\right)^l P_2\\
		+\left(U_{n}-\frac{1}{2}\right)^k \left(U'_{n}+\frac{1}{2}\right)^l P_3+\left(U_{n}+\frac{1}{2}\right)^k \left(U'_{n}+\frac{1}{2}\right)^l P_4\big),
	\end{multline*}
	where $P_1,\,P_2,\, P_3$ and $P_4$ are as in Lemma~\ref{l:submartingaleA} with $ D_n=U_n+\frac{n}{2} $ and $ D'_n=U'_n+\frac{n}{2} $.
	The right hand side of this equation only depends on $\EE(U_{n}^i {U'_{n}}^j)$ with $i\leq k$ and $j\leq l$.
	Hence, inductively computing $\EE(U_{n}^i {U'_{n}}^j)$ for all pairs $(i,j)$ with  $i\leq k$, $j\leq l$ and where at least one of this inequalities is strict, we obtain a recursion formula for $\EE(U_{n}^k {U'_{n}}^l)$.
	
	To obtain the claimed result, we computed the starting values with \sage and solved the recursion with the \textbf{RSolve} command of \Mathematica, just as in Section~\ref{s:fourthmomentD}. The intermediate results of these computations can be found in Appendix~\ref{app:recursionmomentsA}.
\end{proof}

\subsection{Moments of T for type \texorpdfstring{$\cox{B}_n$}{B}}
\label{s:fourthmomentTCoxB}
We now turn to type $\cox{B}_n$. Let $D_n\coloneqq D_{\cox{B}_n}$, $T_n\coloneqq T_{\cox{B}_n}$ and let $D_n'$ be the random variable associated to
\begin{align*}
	\cox{B}_n &\to \mathbb{N}\\
	w &\mapsto \des(w^{-1}).
\end{align*}
To compute the fourth central moment of $T_n=D_n+D'_n$, we want to take the same approach as in Section~\ref{s:fourthmomentTCoxA}. For this, we first need an analogue of Lemma~\ref{l:submartingaleA}.
We start by setting
\begin{align*}
	B_{n,i,j} \coloneqq \left|  \lbrace  w\in \cox{B}_n \, \middle|\, \des(w)=i \text{ and } \des(w^{-1})=j \rbrace \right|.
\end{align*}
These numbers are the coefficients of the \emph{type $\cox{B}_n$ two-sided Eulerian polynomial}
\begin{align*}
	B_n(s,t) \coloneqq \sum_{w\in \cox{B}_n} s^{\des(w)} t^{\des(w^{-1})},
\end{align*}
as studied by Visontai in \cite{Vis:Someremarksjoint}.
We clearly have
\begin{align*}
	\mathbb{P}(\,(D_n,D'_n)=(i,j)\,)= \frac{B_{n,i,j}}{|\cox{B}_n|}.
\end{align*}
\begin{lemma}
	\label{l:recursionBn}
	The numbers $B_{n,i,j}$ satisfy the following recursion formula:
	\begin{align}
		\begin{split}
			\label{eq:recursionBn}
			nB_{n,i,j}= & (n+i+j+2ij)B_{n-1,i,j} \\
			&+ (1-i+(2n+1)j - 2ij)B_{n-1,i-1,j}\\
			& +(1-j+(2n+1)i - 2ij)B_{n-1,i,j-1} \\
			&+ ( n(2n+3)-(2n+1)i - (2n+1)j + 2ij)B_{n-1,i-1,j-1}.
		\end{split}
	\end{align}
\end{lemma}
\begin{proof}
	In \cite[Theorem 15]{Vis:Someremarksjoint}, Visontai shows that the type $\cox{B}_n$ two-sided Eulerian polynomial satisfies 
	\begin{align*}
		nB_n(s,t)= & (2n^2st-nst+n)B_{n-1}(s,t) \\
		&+ (2nst(1-s)+s(1-s)(1-t))\frac{\partial}{\partial s} B_{n-1}(s,t)\\
		& +(2nst(1-t)+t(1-s)(1-t)) \frac{\partial}{\partial t} B_{n-1}(s,t) \\
		&+ 2st(1-s)(1-t)\frac{\partial^2}{\partial s \partial t}B_{n-1}(s,t).
	\end{align*}
	From this, \eqref{eq:recursionBn} follows by computing the derivatives and comparing the coefficients on both sides.
\end{proof}

Using this, we obtain the following analogue of Lemma~\ref{l:submartingaleA}:
\begin{lemma}\label{l:submartingaleB}
	In type $\cox{B}_n$, the conditional law of $ (D_{n+1},D'_{n+1})$ given $(D_{n},D'_{n}) $ is 
	\begin{align*}
		\EE((D_{n+1},D'_{n+1}) |(D_{n},D'_{n}))&=\begin{cases}
			\left(D_{n},D'_{n}\right) & \text{with prob.} ~ P_1=\frac{n+1+D_{n}+D'_{n}+2D_{n}D'_{n}}{2\left(n+1\right)^2},\\
			\left(D_{n}+1,D'_{n}\right) & \text{with prob.} ~ P_2= \frac{-D_{n}+\left(2n+1\right)D'_{n}-2D_{n}D'_{n}}{2\left(n+1\right)^2},\\
			\left(D_{n},D'_{n}+1\right) & \text{with prob.} ~P_3= \frac{\left(2n+1\right)D_{n}-D'_{n}-2D_{n}D'_{n}}{2\left(n+1\right)^2},\\
			\left(D_{n}+1,D'_{n}+1\right) & \text{with prob.}~  P_4= \frac{\left(2n+1\right)\left(n+1-\left(D_{n}+D'_{n}\right)\right) + 2D_{n}D'_{n}}{2\left(n+1\right)^2}.\\
		\end{cases}  
	\end{align*}
\end{lemma}
As in (\ref{eq:recursionB}), the signed permutation of length $ n $ corresponding to $ (D_n,D_n') $ is generated from the signed permutation of length $ n-1 $ corresponding to $ (D_{n-1},D'_{n-1}) $ by inserting $ n $ multiplied with a binary random variable that assigns equal probability to $ \{\pm 1\} $ in a signed permutation of length $ n-1 $.	
\begin{proof}
	Dividing both sides of \eqref{eq:recursionBn} by $n2^n n!$, we obtain
	\begin{align*}
		\frac{B_{n,i,j}}{|\cox{B}_n|}= & \frac{n+i+j+2ij}{2n^2} \frac{B_{n-1,i,j}}{|\cox{B}_{n-1}|} \\
		&+ \frac{1-i+(2n+1)j - 2ij)}{2n^2}\frac{B_{n-1,i-1,j}}{|\cox{B}_{n-1}|}\\
		& +\frac{1-j+(2n+1)i - 2ij)}{2n^2}\frac{B_{n-1,i,j-1}}{|\cox{B}_{n-1}|} \\
		&+ \frac{n(2n+3)-(2n+1)i - (2n+1)j + 2ij}{2n^2} \frac{B_{n-1,i-1,j-1}}{|\cox{B}_{n-1}|},
	\end{align*}
	where we used that $|\cox{B}_n|=2^n n!$.
	From this, the result follows because, as noted above, we have
	\begin{align*}
		\frac{B_{n,i,j}}{|\cox{B}_n|} = \mathbb{P}\big((D_n,D'_n)=(i,j)\big) && \text{ and } && \frac{B_{n-1,k,l}}{|\cox{B}_{n-1}|} = \mathbb{P}\big((D_{n-1},D'_{n-1})=(k,l)\big),
	\end{align*}
	and with the law of total probability, we derive the conditional probabilities. 
\end{proof}

\begin{proposition}\label{p:fourthmomentB} 
	In type $\cox{B}_n$, $ n\geq 4 $, the fourth central moment of $T_n$ is given by
	\begin{align*}
		\EE((T_n-\EE(T_n))^4)=\frac{1}{60} \left(5 n^2+39 n+79\right)+\frac{2 n-1}{4n(n-1)}.
	\end{align*}
\end{proposition}

\begin{proof}
	The proof is completely analogous to the one of Proposition~\ref{p:fourthmomentA}. Again, we set $U_n\coloneqq D_n-\EE(D_n)$ and $U'_n\coloneqq D'_n-\EE(D'_n)$ such that $T_n-\EE(T_n)=U_n+U'_n$ and observe that it suffices to compute $\EE(U_{n}^k {U'}_{n}^l)$ for all $0\leq k,l \leq 4$ with $k+l=4$.
	This can be done inductively using the recursion formula
	\begin{multline*}
		\EE\left(U_{n+1}^k ({U'_{n+1}})^l\right)=\EE\big(\left(U_{n}-\frac{1}{2}\right)^k \left(U'_{n}-\frac{1}{2}\right)^l P_1+\left(U_{n}+\frac{1}{2}\right)^k \left(U'_{n}-\frac{1}{2}\right)^l P_2\\
		+\left(U_{n}-\frac{1}{2}\right)^k \left(U'_{n}+\frac{1}{2}\right)^l P_3+\left(U_{n}+\frac{1}{2}\right)^k \left(U'_{n}+\frac{1}{2}\right)^l P_4\big),
	\end{multline*}
	where $P_1,\,P_2,\, P_3$ and $P_4$ are as in Lemma~\ref{l:submartingaleB} with $ D_n=U_n+\frac{n}{2} $ and $ D'_n=U'_n+\frac{n}{2} $. We solved the corresponding recursions with the \textbf{RSolve} command of \Mathematica; intermediate results can be found in Appendix~\ref{app:recursionmomentsB}.
\end{proof}

\subsection{Proof of \texorpdfstring{Theorem~\ref{t:fourthmoments}}{Theorem 7}}
\label{s:prooffourthmoments}
We are now able to prove Theorem~\ref{t:fourthmoments}:
\begin{proof}[Proof of Theorem~\ref{t:fourthmoments}.]
	For type $\cox{A}_n$ and $\cox{B}_n$, we obtained the result in Proposition~\ref{p:fourthmomentA} and Proposition~\ref{p:fourthmomentB}, respectively. For type $\cox{D}_n$, we exploit the similarity of $ \cox{B}_n $ and $ \cox{D}_n $ to bound the difference between the respective fourth moments. The group $ \cox{B}_n $ has a more combinatorial description as a group of signed permutations: It is isomorphic to the group of all mappings $ \tilde{\pi}:\{\pm 1,\ldots, \pm n \}\rightarrow \{\pm 1,\ldots, \pm n \} $ such that $ \tilde{\pi}(-i)=-\tilde{\pi}(i) $ (for further details, see \cite[Chapter 8]{BB:CombinatoricsCoxetergroups}). Choosing an element of $\cox{B}_n$ uniformly at random hence is equivalent to choosing a random permutation $\pi\in \on{Sym}(n)$ together with a tuple $(b_1,\ldots , b_n)\in \{\pm 1\}^n$---we then obtain $\tilde{\pi}\in \cox{B}_n$ by setting $\tilde{\pi}(i)\coloneqq b_i \cdot \pi(i)$.
	In this description, $\cox{D}_n$ is the subgroup of $\cox{B_n}$ given by all signed permutations $\tilde{\pi}$ such that $|\{ i\in \{ 1,\ldots, n \} \mid \tilde{\pi}(i)<0   \}|$ is an even number. Choosing an element of $\tilde{\pi} \in \cox{D}_n$ uniformly at random is equivalent to choosing a random permutation $\pi\in \on{Sym}(n)$ together with a tuple $(b_1,\ldots , b_{n-1})\in \{\pm 1\}^{n-1}$ and setting 
	\begin{align*}
		\tilde{\pi}(i)\coloneqq
		\begin{cases}
			b_i \cdot \pi(i) &,\, 1\leq i \leq n-1 \\
			(\prod_{j=1}^{n-1} b_j)\cdot \pi(i) &, i=n.
		\end{cases}
	\end{align*}
	
	These considerations imply that we can write 
	\[ T_{\cox{D}_n}\stackrel{d}{=}T_{\cox{B}_n}+Y_n,\]
	where $ Y_n $ is a bounded random variable (cf. \cite[Proof of Theorem 3]{Roettger2019}). 
	Using the Minkowski inequality, we obtain
	\begin{align*}
		\EE\left((T_{\cox{D}_n}-\EE(T_{\cox{D}_n}))^4 \right)&\le \left(\left(\EE\left((T_{\cox{B}_n}-\EE(T_{\cox{B}_n})^4\right)\right)^{\frac{1}{4}}+O(1)\right)^{4}= \EE\left((T_{\cox{B}_n}-\EE(T_{\cox{B}_n})^4 \right)+O\big(n^{\frac{3}{2}}\big).
	\end{align*}
	The result now follows from Proposition~\ref{p:fourthmomentB}.
\end{proof}

\begin{remark}
	The results of this section show the convenience of the conditional expectation to compute the expected value: Instead of a combinatorial approach as for example in the proof of \cite[Proposition 5.7]{KahleStump2018}, one derives a recursion formula and uses a recursion solver program like \textbf{RSolve} to find the solution. Of course, this approach is only possible if one can find a conditional expectation as for example in Lemma~\ref{l:submartingaleB}.
\end{remark}

\begin{remark}
	In \cite[Section 5.7]{oezdemir2019} it is shown how to derive the CLT for $ T $ in the case $ (W_n)_n=(\cox{A}_n)_{n} $ via the martingale convergence theorem and the recursive formulation of Lemma~\ref{l:submartingaleA}. This is an alternative proof of \cite[Theorem 1.1]{chatterjee2017} and one should be able to find an alternative proof for \cite[Theorem 2]{Roettger2019}, i.e.~to prove the CLT for $ T $ when $ (W_n)_n=(\cox{B}_n)_{n} $ with the given formulas for the moments of $ T_{\cox{B}} $.
\end{remark}

\section{CLTs for weighted sums of converging sequences} \label{s:rank=O(n)}

This section explains how to derive the asymptotic normality of a sequence of random variables $ (X_n)_n $, where $ X_n=\sum_{i=1}^{k_n}a_{n,i}X_{n,i} $, under the assumption that $ (X_{n,i})_n \stackrel{D}{\rightarrow}N(0,1) $ for all $i$.
The main idea is to use Lévy's continuity theorem via the pointwise convergence of the characteristic function of $ X_n $ towards the characteristic function of the standard normal distribution.
We begin with some preparations:
\begin{definition}
	The \emph{characteristic function} of a random variable $ X $ is defined as $ \psi_X(s):=\mathbb{E}\left(e^{isX}\right) $ for $ s\in \mathbb{R} $.
\end{definition}
For a detailed introduction to characteristic functions, see for example \cite{billingsley1995probability}. Now, Lévy's continuity theorem states the following:
\begin{theorem}[Lévy] \label{t:levi}
	For a sequence of random variables $ (X_n)_n $, it holds that $ X_n\stackrel{D}{\rightarrow}X $ for some random variable $ X $ if and only if $ \lim\limits_{n\rightarrow \infty}\psi_{X_n}(s)=\psi_X(s) $ for every $ s \in \mathbb{R} $.
\end{theorem}
Characteristic functions of sums of independent random variables exhibit the following useful property:
\begin{lemma} \label{l:charfactorization} Let $ X $ and $ Y $ be real-valued random variables. If $ X $ and $ Y $ are independent and $ a,b \in \mathbb{R} $, it holds that $ \psi_{aX+bY}(s)=\psi_X(as)\psi_Y(bs) $ for every $ s\in \mathbb{R} $.
\end{lemma} 

Using the preceding results, one obtains the following lemma, which describes when a weighted sum of converging sequences satisfies the CLT.
Note that in the following, the array $ (X_{n,i})_{n,1\le i\le k_n} $ is not required to be triangular.

\begin{lemma} \label{l:infinitesum}
	For each $ n\in \mathbb{N} $, let $k_n\in \mathbb{N}_{>0}$ be a positive natural number. Let $a_{n,i}\in \mathbb{R}_{\geq 0}$,   $1\leq i \leq k_n$, such that $\sum_{i=1}^{k_n}a_{n,i}^2=1$ and let $ X_{n,i} $, $1\leq i \leq k_n$, be independent centred random variables with $ \VV(X_{n,i})=1 $.
	Define $ X_n= \nolinebreak\sum_{i=1}^{k_n}a_{n,i}X_{n,i} $. 
	Then if for each $i$, we have $ X_{n,i} \stackrel{D}{\rightarrow}N(0,1) $ and 
	\begin{align}
		\label{eq:conditionwellbehaved}
		\lim\limits_{k\rightarrow \infty}\sup\limits_n \left(\sum\limits_{i=k}^{k_n}a_{n,i}^2 \right) =0 ,
	\end{align}
	it follows that $ X_n\stackrel{D}{\rightarrow} N(0,1) $.
\end{lemma}

Before proving this, we give some comments on \eqref{eq:conditionwellbehaved}.
Let $X_n^k\coloneqq \sum_{i=1}^{\min(k,k_n)}a_{n,i}X_{n,i}$ be the random variable that is given by as the sum of the first $k$ summands of $X_n$. We have  $\VV(X_n) = \sum_{i=1}^{k_n}a_{n,i}^2=1$ and
\begin{gather*}
	\VV(X_n^k) = \sum\limits_{i=1}^{\min(k,k_n)}a_{n,i}^2 = 1- \sum\limits_{i=k}^{k_n}a_{n,i}^2.
\end{gather*}
Hence, \eqref{eq:conditionwellbehaved} is equivalent to
\begin{gather*}
	\lim\limits_{k\rightarrow \infty}\sup\limits_n \left(\VV(X_n) - \VV(X_n^k) \right) = 0.
\end{gather*}
This means that the statement of Lemma~\ref{l:infinitesum} can roughly be phrased as follows: If all the columns of the array $(X_{n,i})_{n\in \mathbb{N},1\le i\le k_n}$ satisfy the CLT and furthermore, the initial summands of $X_n$ asymptotically contain all of the variance of $X_n$, then $(X_n)_n$ satisfies the CLT.

\begin{proof}[Proof of Lemma~\ref{l:infinitesum}]
	The characteristic function of the normal distribution is $ e^{-\frac{1}{2}s^2} $. To prove the asymptotic normality of $ X_n $, we therefore show that for all $ s\in\mathbb{R} $ and any $ \delta>0 $, there is an $ N\in \mathbb{N} $ so that $ |\psi_{X_n}(s)-e^{-\frac{1}{2}s^2}|<\delta $  for all $ n\ge N $. Now,
	\begin{align*}
		|\psi_{X_n}(s)-e^{-\frac{1}{2}s^2}|
		&\le |\psi_{X_n}(s)-\psi_{\sum_{i=1}^{k}a_{n,i}X_{n,i}}(s)|+|\psi_{\sum_{i=1}^{k}a_{n,i}X_{n,i}}(s)-e^{-\frac{1}{2}s^2}|.
	\end{align*}
Condition~\eqref{eq:conditionwellbehaved} guarantees that for any $ \varepsilon>0 $, there is a finite $ k $ such that for all $n$, one has $ \sum\limits_{i=k+1}^{\infty}a_{n,i}^2\le \varepsilon $.
	We conclude for the first summand with Jensen's inequality and $ |e^{i\alpha}-1|\le |\alpha| $ that
	\begin{align*}
		|\psi_{X_n}(s)-\psi_{\sum_{i=1}^{k}a_{n,i}X_{n,i}}(s)|&=|\EE(e^{isX_n}-e^{is\sum_{i=1}^{k}a_{n,i}X_{n,i}})|\\
		&\le \EE|e^{is\sum_{i=k+1}^{\infty}a_{n,i}X_{n,i}}-1|\\
		&\le \EE|s\sum_{i=k+1}^{\infty}a_{n,i}X_{n,i}|\\
		&\le|s|\left(\EE\left(\sum_{i=k+1}^{\infty}a_{n,i}X_{n,i}\right)^2\right)^{\frac{1}{2}} \le |s|\left(\sum_{i=k+1}^{\infty}a_{n,i}^2\right)^{\frac{1}{2}} \le |s|\varepsilon^{\frac{1}{2}}.
	\end{align*}
	For the second summand, with the uniform convergence of characteristic functions on compact intervals and the asymptotic normality of $ (X_{n,i})_n $, i.e. $ \psi_{X_{n,i}}(s)\rightarrow e^{-\frac{1}{2}s^2} $ , we obtain for some positive constants $ C_1,C_2 $
	\begin{align*}
		|\psi_{\sum_{i=1}^{k}a_{n,i}X_{n,i}}(s)-e^{-\frac{s^2}{2}}|
		&\le |\prod_{i=1}^{k}\psi_{X_{n,i}}(a_{n,i}s)-\prod_{i=1}^{k}e^{-a_{n,i}^2\frac{s^2}{2}}| + |e^{-\sum_{i=1}^{k}a_{n,i}^2\frac{s^2}{2}}-e^{-\frac{s^2}{2}}|\\
		&\le C_1\varepsilon +|e^{-\frac{s^2}{2}}(e^{-(1-\sum_{i=1}^{k}a_{n,i}^2)\frac{s^2}{2}}-1)|\\
		&\le C_1\varepsilon +|e^{-\frac{s^2}{2}}(e^{-\varepsilon\frac{s^2}{2}}-1)|\le C_2 \varepsilon.
	\end{align*}	
	These considerations imply that for any $ \varepsilon>0 $ and some positive constant $ C_3(s) $, there is an $ N\in \mathbb{N} $ so that for all $ n\ge N $ it holds that $ |\psi_{X_n}(s)-e^{-\frac{1}{2}s^2}|\le C_3(s)\varepsilon=\delta $.
\end{proof}

The following lemma is a consequence of Lemma~\ref{l:infinitesum} when $ k_n $ is globally bounded, but additionally allows for summands that converge in probability towards zero, instead of converging in distribution to the standard normal distribution.

\begin{lemma}
	\label{l:CLTforSumofIndependent}
	Let $(X_n)_n$ be a sequence of centred random variables and suppose that there is $k\in \mathbb{N}$ such that for each $n$, $X_n$ can be written as a sum $X_n = X_{n,1} + \cdots + X_{n,k}$ of independent random variables $X_{n,i}$. Assume that for every $1\leq i \leq k$, the following holds true: Either $(X_{n,i})_n$ satisfies the CLT or $\frac{X_{n,i}}{\sqrt{\VV(X_{n})}}\stackrel{\PP}{\rightarrow} 0$. Then if at least one sequence $(X_{n,i})_n$ satisfies the CLT and $ \VV(X_{n}) \to \infty$, the sequence $(X_n)_n$ satisfies the CLT.
\end{lemma}
\begin{proof}
	Without loss of generality, we can assume that there is $ k'\geq 1$ such that for $1\leq i\leq k'$, the sequence $(X_{n,i})_n$ satisfies the CLT while for all $i> k'$, we have $\frac{X_{n,i}}{\sqrt{\VV(X_{n})}}\stackrel{\PP}{\rightarrow} 0$. This implies that 
	\begin{align*}
		Z_n\coloneqq \frac{X_{n,k'+1} + \cdots + X_{n,k}}{\sqrt{\VV(X_{n})}}\stackrel{\PP}{\rightarrow} 0
	\end{align*}	
	Using Slutsky's Theorem \cite[Theorem 2.3.3]{Lehmann1998}, we see that $ X_{n} $ satisfies the CLT if the remaining sum $ X_n' = X_n-Z_n = X_{n,1} + \cdots + X_{n,k'}$ satisfies the CLT.
	We can write
	\begin{align*}
		\frac{X_n'}{\sqrt{\VV (X_n')}}= \sum_{i=1}^{k'} a_{n,i} \frac{X_{n,i}}{\sqrt{\VV(X_{n,i})}},\hspace{0.5 cm} \text{ where } \hspace{0.5 cm} a_{n,i} = \sqrt{\frac{\VV(X_{n,i})}{\VV(X_n')}} .
	\end{align*}
	We have
	\begin{align*}
		\sum_{i=1}^{k'}a_{n,i}^2 = \frac{\sum_{i=1}^{k'} \VV(X_{n,i})}{\VV(X_n')} = 1, 
	\end{align*}	 
	so the claim follows from Lemma~\ref{l:infinitesum} as \eqref{eq:conditionwellbehaved} is trivially satisfied.
\end{proof}

\begin{lemma} \label{l:CLTforSumofIndependent2}
	In the setting of Lemma~\ref{l:CLTforSumofIndependent}, the condition  $\frac{X_{n,i}}{\sqrt{\VV(X_{n})}}\stackrel{\PP}{\rightarrow} 0$ holds if $ \frac{\VV(X_{n,i})}{\VV(X_{n})}\rightarrow~0 $.
\end{lemma}
\begin{proof}
	The Chebyshev inequality shows that
	\begin{align*}
		\PP\left(\frac{|X_{n,i}|}{\sqrt{\VV(X_n)}}\ge \varepsilon\right)&\le \frac{\VV(X_{n,i})}{\varepsilon^2 \VV(X_n)},
	\end{align*}
	which implies the convergence in probability of $ \frac{|X_{n,i}|}{\sqrt{\VV(X_n)}} $ towards zero if $ \frac{\VV(X_{n,i})}{\VV(X_{n})}\rightarrow 0 $.
\end{proof}

\section{CLT via the Lindeberg Theorem} \label{s:Lindeberg}
A collection $\left(X_{n,i}\right)_{n\ge 1}^{1\le i \le k_n}$ of random variables is called a  \emph{triangular array} if for each $n$, all $X_{n,i}$ are independent of each other. A triangular array is called \emph{centred} if $\EE(X_{n,i})=0$ for all $n$ and $i$. Given such a triangular array, we set 
\begin{align*}
	X_n\coloneqq \sum_{i=1}^{k_n}X_{n,i},&& s_{n,i}^2\coloneqq \VV(X_{n,i}) && \text{and} && s_n^2\coloneqq \VV(X_n)=\sum_{i=1}^{k_n}s^2_{n,i}.
\end{align*} 
The array $(X_{n,i})_{n,i}$ satisfies the \emph{maximum condition} if 
\begin{align}
	\lim_{n \rightarrow \infty} \max_{1\le i \le k_n} \frac{s_{n,i}^2}{s_{n}^2}=0.\label{eq:maximum_condition}
\end{align}
It satisfies the \emph{Lindeberg condition} if for every $ \varepsilon>0 $,
\begin{align*}
	\frac{1}{s_n^2}\sum_{i=1}^{k_n}\mathbb{E}\left(X_{n,i}^2\mathbbm{1}_{\{|X_{n,i}|>\varepsilon s_n\}}\right)\rightarrow 0,
\end{align*}
where $\mathbbm{1}_{\lbrace \cdot \rbrace}$ denotes the indicator function. 
The significance of these conditions for us is as follows:
\begin{theorem}[Lindeberg]
	\label{t:Lindeberg}
	Let $(X_{n,i})_{n,i}$ be a centred triangular array. Then $(X_{n,i})_{n,i}$ satisfies the Lindeberg condition if and only if it satisfies the maximum condition and the sequence $(X_n)_n$ satisfies the CLT.
\end{theorem} 
The Lindeberg condition is implied by the \emph{Lyapunov condition}, which is satisfied if for some $ \delta>0 $ it holds that
\begin{align*}
	\frac{1}{s_n^{2+\delta}} \sum_{i=1}^{k_n}\mathbb{E}\left(|X_{n,i}|^{2+\delta}\right) \rightarrow 0.
\end{align*}
To apply this to our setting, let $(W_n)_n$ be a sequence of finite Coxeter groups and let 
\begin{align*}
	W_n&= \prod_{i=1}^{k_n} W_{n,i},
\end{align*}
be the decomposition of $W_n$ into its irreducible components. Now, let $T_n$ be the random variable associated to the statistic $t$ on $W_n$. By Lemma~\ref{l:sumofindependent}, we have
\begin{align*}
	T_n= \sum_{i=1}^{k_n} T_{n,i},
\end{align*}
where $T_{n,i}$ is the random variable associated to the statistic $t$ on $W_{n,i}$.
From this, we obtain a centred triangular array by setting $X_{n,i}\coloneqq T_{n,i}-\EE(T_{n,i})$. By the arguments above, we have $X_n=T_{n}-\EE(T_{n})$.

\begin{lemma} \label{l:CLTviaLindeberg}
	Let $(W_n)_n$ be a sequence of finite Coxeter groups such that $\VV(T_{n,1})\ge \ldots \ge \VV(T_{n,k_n})$ for all $n$ and such that
	$\VV(T_{n,1}) = o(\VV(T_{n}))$ and $ \VV(T_{n})\to \infty $.
	Then $ (T_n)_n $ satisfies the CLT.
\end{lemma}
\begin{proof}
	As above, let $(X_{n,i})_{n,i}$ be the triangular array associated to the sequence $(W_n)_n$. 
	We want to apply the Lindeberg Theorem. The maximum condition is satisfied by assumption, so we only need to verify the Lindeberg condition. We do so via the Lyapunov condition. 
	To check the Lyapunov condition, we choose $ \delta=2 $. We see that $ \EE(X_{n,i}^4)=O(s_{n,i}^4) $ for the non-dihedral infinite families (cf.~Theorem~\ref{t:fourthmoments}).
	If $W_{n,i}$ is of dihedral or exceptional type, $|X_{n,i}|$ is globally bounded: This is clear for the finitely many exceptional types. For $w\in \cox{I}_2(m_{n,i})$, it is easy to verify that 
	\begin{align*}
		0\leq t(w)=\des(w)+\des(w^{-1})\leq 4.
	\end{align*}
	We have $\rk(\cox{I}_2(m_{n,i}))=2$, so by Theorem~\ref{t:varianceofT}, one has 
	\begin{align*}
		|X_{n,i}|=|T_{n,i}-\EE(T_{n,i})|\leq 2. 
	\end{align*}
	Therefore, the fourth moment of the dihedral or exceptional	type is bounded by a constant, so $ \EE(X_{n,i}^4)=O(1)=O(s_{n,i}^4) $. 
	Now, as $ s_{n,1}^2=o(s_n^2) $ and $ s_n^2=\sum_{i=1}^{k_n}s_{n,i}^2 $, the Lyapunov condition holds, because
	\begin{align*}
		\sum_{i=1}^{k_n}\mathbb{E}\left(|X_{n,i}|^{4}\right)&=O\left(\sum_{i=1}^{k_n}s_{n,i}^4\right)=O(s_{n,1}^2 s_n^2).
	\end{align*}
\end{proof}

\section{Proof of the main theorem} \label{s:mainresult}
Throughout this section, let $(W_n)_n$ be a sequence of finite Coxeter groups such that $\rk(W_n)\to \infty$, let 
\begin{align*}
	W_n&= \prod_{i=1}^{k_n} W_{n,i}
\end{align*}
be the decomposition of $W_n$ into its irreducible components and define $T_n\coloneqq T_{W_{n}}$ and $T_{n,i}\coloneqq T_{W_{n,i}}$.

\begin{assumption}\label{as:ordered}
	We assume that the irreducible components are ordered such that for all $n$, we have $\VV(T_{n,1})\ge \ldots \ge \VV(T_{n,k_n})$.
\end{assumption}

In the previous section, we proved the CLT for sequences where the variance of $T_{n,i}$ was of smaller magnitude than the variance of $T_{n}$  (Lemma~\ref{l:CLTviaLindeberg}).
However, this need not be the case in general; if the $W_{n,i}$ are of non-dihedral type, it is possible that for some $i$, the rank of $W_{n,i}$ is of the same order as the rank of $W_n$. An easy example of this is given by setting $W_n\coloneqq \prod_{i=1}^{k} \cox{A}_n$ for some $k\in \mathbb{N}$; here, we have $\VV(T_n)/ \VV(T_{n,i})=k$ for all $n$. An example with a growing number of irreducible components is the sequence $ W_n=\prod_{i=1}^{\lceil \log(n)\rceil}A_{\lceil\frac{n}{2^i}\rceil} $, so that $\VV(T_n)/ \VV(T_{n,i})=2^i $.
In order to extend our results to these cases, we need to separate the irreducible components that do not satisfy the maximum condition~\eqref{eq:maximum_condition} from the remaining ones. For this, we make the following definition:

Let $f:\mathbb{R}_{\geq 0} \to \mathbb{R}_{\geq 0}$ be a non-decreasing map such that $f(n)=o(n)$.
An irreducible component $ W_{n,i} $ of $W_n$ is called \emph{$f$-small}, if $ \VV(T_{n,i})\le f(\VV(T_n))$.
Let 
\begin{equation*}
m_n\coloneqq \min\{i\in \mathbb{N}:~ W_{n,i+1}~ \text{is $f$-small}\}.
\end{equation*}
By Assumption~\ref{as:ordered}, $ W_{n,i} $ is $ f $-small for all $ i\ge m_n $.
We define $M_n^f\coloneqq \prod_{i=1}^{m_n} W_{n,i}$ and $W_n^f \coloneqq \prod_{i=m_n+1}^{k_n} W_{n,i}$. For all $n$, we can write $W_n=M_n^f \times W_n^f$.
By Lemma~\ref{l:sumofindependent}, we have
\[T_n=T_{M_n^f}+T_{W_n^f}=\sum_{i=1}^{m_n} T_{n,i}+\sum_{i=m_n+1}^{k_n} T_{n,i} .\]

\begin{remark}
	\label{r:DihedralExceptionalSmall}
	Among the class of all finite irreducible Coxeter groups $W$ of dihedral or exceptional type, the variance $\VV(T_W)$ is bounded from above: If $W$ is dihedral, then $\VV(T_W)\leq 2$ and there are only finitely many exceptional types. Hence if $\VV(T_n)\to \infty$, then for every  non-decreasing $f:\mathbb{R}_{\geq 0}\to \mathbb{R}_{\geq 0}$ with $f(n)=o(n)$, there is $N\in \mathbb{N}$ such that for all $n\geq N$, every irreducible component of $W_n$ is either of type $\cox{A}, \cox{B}$ or $\cox{D}$ or it is $f$-small.
\end{remark}

As was shown by Chatterjee--Diaconis \cite{chatterjee2017} and  R\"ottger \cite{Roettger2019}, the sequences $ T_{\cox{A}_n},T_{\cox{B}_n} $ and $ T_{\cox{D}_n} $ satisfy the CLT. This allows us to apply Lemma~\ref{l:infinitesum} if the sequence $ (W_n)_n $ satisfies the following property:
\begin{definition}
	The sequence $ \left(W_n\right)_n $ is \emph{well-behaved}, if there exists a non-decreasing function $f:\mathbb{R}_{\geq 0}\to \mathbb{R}_{\geq 0}$ with $f(n)=o(n)$, such that
	\begin{align}
		\label{eq:wellbehavedcondition}
		\lim\limits_{k\rightarrow \infty}\sup\limits_n \left( \sum\limits_{i=k}^{m_n}\frac{\VV(T_{n,i})}{\VV(T_{M_n^f})} \right) = 0.
	\end{align}
\end{definition}
Note that the condition~\eqref{eq:wellbehavedcondition} relates directly to condition~\eqref{eq:conditionwellbehaved} when we are interested in deriving the CLT for $ T_{M_n^f} $ in the case that the $ T_{n,i} $ satisfy the CLT.
\begin{remark}
	While the definition seems to be rather technical, the authors have failed to construct a sequence that is not well-behaved. A reason why it is hard to find such a sequence is the following: 
	
	A sequence is always well-behaved if $m_n$, the number of irreducible components that are not $f$-small, is bounded. This follows because under Assumption~\ref{as:ordered} we have
	\begin{align*}
		\sum\limits_{i=k}^{m_n}\frac{\VV(T_{n,i})}{\VV(T_{M_n^f})} \leq \max\{m_n-k,0\} \cdot \frac{\VV(T_{n,1})}{\VV(T_{M_n^f})}.
	\end{align*}
	That $m_n$ is bounded is for example the case if the rank or the number of irreducible components in $W_n$ are bounded. It is also the case if 
	there is a $J\in \mathbb{N}$ such that for all $i>J$, the sequence of $i$-th components $(T_{n,i})_{n \in \mathbb{N}}$ satisfies $(\VV(T_{n,i}))_{n \in \mathbb{N}} = o((\VV(T_n))_{n \in \mathbb{N}})$. If there are no irreducible components of dihedral type, the latter is satisfied if there is $J\in \mathbb{N}$ such that for all $i>J$, we have $(\rk(W_{n,i}))_{n \in \mathbb{N}} = o((\rk(W_n))_{n \in \mathbb{N}})$ (see Remark~\ref{r:DihedralExceptionalSmall} and Theorem~\ref{t:varianceofT}); in other words, the sequence is well-behaved if there are only boundedly many irreducible components of $W_n$ that have rank not in $o(\rk(W_n))$.
	
	So if one wants to find a sequence that is not well-behaved, one needs $m_n$ to be unbounded. However, even then well-behavedness occured in all examples that the authors considered, see e.g.~Example~\ref{ex:CoxetergroupsandCLT} (i).
\end{remark}

\begin{remark}
	\label{r:subsequencewellbehaved}
	For all $L\subseteq \mathbb{N}$, we obviously have
	\begin{align*}
		\sup\limits_{n\in L} \left( \sum\limits_{i=k}^{m_n}\frac{\VV(T_{n,i})}{\VV(T_{M_n^f})} \right) \leq \sup\limits_{n\in \mathbb{N}} \left( \sum\limits_{i=k}^{m_n}\frac{\VV(T_{n,i})}{\VV(T_{M_n^f})} \right) \text{ for all } k.
	\end{align*}
	Thus, every subsequence of a well-behaved sequence is well-behaved again.
\end{remark}

\begin{proposition}
	\label{p:CLTforNonDihedral}
	If $ \left(W_{n}\right)_n $ is well-behaved and $\VV(T_n)\to \infty$, the sequence $(T_n)_n$ satisfies the CLT.
\end{proposition}

\begin{proof}
	Choose $f$ such that \eqref{eq:wellbehavedcondition} is satisfied.
	As noted above, we have $ T_n=T_{M_n^f}+T_{W_n^f} $, so by assumption $\VV(T_{M_n^f})+ \VV(T_{W_n^f}) = \VV(T_n) \to \infty$.
	
	By Lemma~\ref{l:subsubsequence}, it suffices to show that every subsequence of $(T_n)_{n\in  \mathbb{N}}$ has a subsequence which satisfies the CLT. For any $L\subseteq \mathbb{N}$, the subsequence $(W_n)_{n\in L}$ satisfies all conditions of the proposition: Obviously, the rank $(\rk(W_n))_{n\in L}$ tends to infinity and so does the variance $(\VV(T_n))_{n\in L}$. Furthermore, the sequence is well-behaved as noted in Remark~\ref{r:subsequencewellbehaved}. Hence, it suffices to consider the case where $L = \mathbb{N}$: We will now show that an (arbitrary) sequence $(T_n)_{n\in  \mathbb{N}}$ as in the statement of the proposition has a subsequence that satisfies the CLT. It then follows that every subsequence of $(T_n)_{n\in  \mathbb{N}}$ has a subsequence with this property as well.
	
	If  $\VV(T_{M_n^f})=o(\VV(T_n))$, then $\VV(T_{W_n^f})$ is of the same order as $\VV(T_n)$. Hence, as every irreducible factor of $W_n^f=\prod_{i=m_n+1}^{k_n} W_{n,i}$ is $f$-small, we have $\VV(T_{n,{m_{n+1}}}) = o(\VV(T_{W_n^f}))$. This allows us to apply Lemma~\ref{l:CLTviaLindeberg} to see that $(T_{W_n^f})_n$ satisfies the CLT. The CLT for $(T_n)_n$ now follows---even without passing to a subsequence---from Lemma~\ref{l:CLTforSumofIndependent} and Lemma~\ref{l:CLTforSumofIndependent2} because $\VV(T_{M_n^f}) / \VV(T_n)\to 0$.
	
	Next assume that $\VV(T_{M_n^f})\not=o(\VV(T_n))$. In this case, there is $L\subseteq \mathbb{N}$ such that $(\VV(T_{M_n^f}))_{n\in L}\to \infty$ holds true\footnote{Note that $(\VV(T_{M_n^f}))_{n\in \mathbb{N}}\to \infty$ need not be true. This makes it necessary to pass to a subsequence here---in contrast to the previous paragraph.}. The subsequence $(M_n^f)_{n\in L}$ is again well-behaved and as noted in Remark~\ref{r:DihedralExceptionalSmall}, we can assume that every irreducible component of $M_n^f$ is of type $\cox{A}, \cox{B}$ or $\cox{D}$. Thus, it follows from \cite{chatterjee2017}, \cite{Roettger2019} and Lemma~\ref{l:infinitesum} that the sequence $(T_{M_n^f})_{n\in L}$ satisfies the CLT. 
	The asymptotic normality of $(T_n)_{n\in L}$ now follows from Lemma~\ref{l:CLTforSumofIndependent} and Lemma~\ref{l:CLTforSumofIndependent2}: Either $\VV(T_{W_n^f})$ is of the same order as $\VV(T_n)$; because every component of $W_n^f$ is $f$-small, this implies that after possible passing to a further subsequence, $T_{W_n^f}$ satisfies the CLT. Or we have $\VV(T_{W_n^f}) / \VV(T_n)\to 0$.
\end{proof}

We are now ready to prove our main theorem. Each $W_n$ decomposes uniquely as 
\begin{align*}
	W_n= G_n \times I_n,
\end{align*}
where no irreducible component of $G_n$ is of dihedral type and 
\begin{align*}
	I_n= \prod_{i=1}^{l_n} \cox{I}_2(m_{n,i}).
\end{align*}
Note that by Remark~\ref{r:DihedralExceptionalSmall}, the sequence $(W_n)_n$ is well-behaved if and only if $(G_n)_n$ is. We use this decomposition in order to combine the results obtained so far and show:

\begin{theorem}\label{t:mainresult}
	Let $T_n$ be the random variable associated to the statistic $t$ on $W_n$. Assume that $(W_n)_n$ is well-behaved. Then the following are equivalent:
	\begin{enumerate}
		\item \label{i:CLTsatisfied} $(T_n)_n$ satisfies the CLT;
		\item \label{i:Variancetoinfty} $\VV(T_n)\rightarrow \infty$;
		\item \label{i:Variancesseparately} $\rk(G_n)+\sum_{i=1}^{l_n} \frac{1}{m_{n,i}}\rightarrow \infty $.
	\end{enumerate}
\end{theorem}

\begin{proof}``$(2)\Leftrightarrow (3)$": By Lemma~\ref{l:sumofindependent}, the random variable $T_n$ decomposes as a sum of independent random variables $T_n = T_{G_n} + T_{I_n}$. 
	By Theorem~\ref{t:varianceofT}, $\rk(G_n)$ is of order $\VV(T_{G_n})$ and $\sum_{i=1}^{l_n} \frac{1}{m_{n,i}}$ is of order $\VV(T_{I_n})$.
	Using additivity of the variance, it follows immediately that Item~\ref{i:Variancetoinfty} is equivalent to Item~\ref{i:Variancesseparately}.
	
	\noindent
	``$ (2)\Rightarrow (1) $": That Item~\ref{i:Variancetoinfty} implies Item~\ref{i:CLTsatisfied} is the statement of Proposition~\ref{p:CLTforNonDihedral}.
	
	\noindent
	``$ (1)\Rightarrow(2) $": Lastly, as $T_n-\EE(T_n)$ takes only values in $\mathbb{Z}$, the sequence $(T_n)_n$ can only satisfy a CLT if its variance tends to infinity \cite[Proposition 6.15]{KahleStump2018}. This shows that Item~\ref{i:CLTsatisfied} implies Item~\ref{i:Variancetoinfty}.
\end{proof}

We note that Item~\ref{i:CLTsatisfied} implies Item~\ref{i:Variancetoinfty} even without assuming that the sequence is well-behaved.

\begin{example}
	\label{ex:CoxetergroupsandCLT}
	The following list of examples illustrates Theorem~\ref{t:mainresult}. To simplify the notation, we omit the rounding of the ranks of the irreducible components and write $ W^k=\prod_{i=1}^{k}W $ for the product of $ k $ copies of the group $ W $.
	\begin{itemize}
		\item[(i)] $ W_n=\prod_{i=1}^{\log(n)}\cox{A}_{\frac{n}{2^i}}\times (\cox{B}_{\sqrt{n}})^{\sqrt{n}} $ satisfies the CLT:  $ (\cox{B}_{\sqrt{n}})^{\sqrt{n}} $ satisfies \eqref{eq:maximum_condition}. We need to show that the first factor $ \prod_{i=1}^{\log(n)}\cox{A}_{\frac{n}{2^i}} $ is well-behaved.
		Note that $ m_n=\log(n) $.	We have $ \VV(T_{\cox{A}_n})=\frac{n}{6}+O(1) $, such that
		\begin{align*}
			\sum\limits_{i=k}^{m_n}\frac{\VV(T_{n,i})}{\VV(T_{M_n^f})}&=\frac{\sum\limits_{i=k}^{m_n}\VV(T_{n,i})}{\sum\limits_{i=1}^{m_n}\VV(T_{n,i})} =\frac{\sum\limits_{i=k}^{m_n}(\frac{n}{ 2^i}+O(1))}{\sum\limits_{i=1}^{m_n}(\frac{n}{2^i}+O(1))}=\frac{\sum\limits_{i=k}^{m_n}2^{-i}+o(1)}{\sum\limits_{i=1}^{m_n}2^{-i}+o(1)}.
		\end{align*}
		As the geometric series converges, $ \lim\limits_{k\to\infty}\sup\limits_{n} $ of the above goes to zero and therefore the sequence $ (W_n)_n $ is well-behaved.
		\item[(ii)]  For any $ 0<\delta<1 $, the product $ \cox{B}_{n}\times (\cox{A}_{n^{1-\delta}})^{n^{\delta}} $  satisfies the CLT: Define $f(n)\coloneqq n^{1-\delta}$. Then $ m_n=1 $ is bounded and $ (\cox{A}_{n^{1-\delta}})^{n^{\delta}} $ satisfies \eqref{eq:maximum_condition}.
		\item[(iii)] $ W_n=\prod_{i=1}^{n}I_2(i) $ satisfies the CLT, as the harmonic series diverges and therefore $ \VV(T_n)\to \infty $.
		\item[(iv)] $ W_n=\prod_{i=1}^{n}I_2(i^2) $ does not satisfy the CLT, as $ \VV(T_n) $ is bounded.
		\item[(v)] $ W_n=\cox{A}_3^n \times \cox{D}_5^{n}\times \cox{F}_4^n\times \prod_{i=1}^{n}I_2(i^2) $ satisfies the CLT, as $ m_n=0 $ and $ \VV(T_n)\to \infty $.
		Note that $ \cox{F}_4 $ is a Coxeter group of exceptional type.
	\end{itemize}
\end{example}

\begin{appendix}
	\section{Geometric interpretation of \texorpdfstring{$t$}{t}}
	\label{appendixGeometry}
	Throughout this section, let $(W,S)$ be a fixed Coxeter system and let $n\coloneqq |S|$ be its rank.
	In this section, we give an interpretation of the statistic
	\begin{align*}
		t: W & \to \mathbb{N}\\ 
		w &\mapsto  \des(w)+ \des(w^{-1}).
	\end{align*}
	in terms of a boolean complex defined by Hultman \cite{Hultman2007}. We here use the same notation as Petersen in \cite{Pet:twosidedanalogue}.
	
	Associated to $W$ is its \emph{Coxeter complex} $\Sigma=\Sigma(W,S)$, a simplicial complex which is defined as follows:
	For $I\subseteq S$, denote by $W_I$ the (parabolic) subgroup of $W$ generated by $I$. The faces of $\Sigma$ are given by all cosets $w W_I$, where $w\in W$ and $I\subseteq S$; the face relation is defined by
	\begin{align*}
		w W_I \leq_\Sigma w' W_{I'} \hspace{0.5cm}\text{ if and only if }\hspace{0.5cm}
		w W_I \supseteq w' W_{I'}.
	\end{align*}
	Coxeter complexes are classical, well-studied structures that give a geometric way of investigating properties of Coxeter groups and related structures; for further details, see e.g. \cite[Chapter 3]{AB:Buildings}.
	
	In \cite{Hultman2007}, Hultman defines a complex $\Xi=\Xi(W,S)$, which can be seen as a \emph{two-sided Coxeter complex}. The faces of $\Xi$ are given by all triples $(I, W_I w W_J, J)$, where $I, J\subseteq S$, $w\in W$ and $W_I w W_J$ denotes the corresponding double coset. The face relation is given by
	\begin{align*}
		(I, W_I w W_J, J) \leq_\Xi (I', W_{I'} w' W_{J'}, J')\hspace{0.5cm}\text{ if and only if }\hspace{0.5cm}
		\begin{cases}
			I\supseteq I',\\
			J\supseteq J' \text{ and } \\
			W_I w W_J \supseteq W_{I'} w' W_{J'}.
		\end{cases}
	\end{align*}
	Petersen \cite{Pet:twosidedanalogue} showed that $\Xi$ shares several properties with $\Sigma$: It is a balanced, shellable complex and if $W$ is finite, the geometric realisation of $\Xi$ is homeomorphic to a sphere of dimension $2n-1$.
	A difference between the two structures is that $\Xi$ is not a simplicial, but only a boolean complex.
	A \emph{boolean complex} (or \emph{simplicial poset}) is a poset $P$ with a unique minimal element $\hat{0}$ such that every lower interval $[\hat{0},p]$ is a boolean algebra, i.e. equivalent to the face poset of a simplex. Such a poset can also be seen as a semi-simplical set; its maximal faces (or \emph{facets}) are the maximal elements of $P$ and the face maps are induced by the partial order of $P$. Using this description, the vertices are the minimal elements of $P \setminus \{\hat{0}\}$. The face poset of a simplicial complex is an example of a boolean complex. The complex $\Xi$ however is not simplicial---in fact, all of its facets share the same vertex set.

	From now on, we assume that $W$, and hence $\Xi$, is finite.
	The statistic $t$ has two interpretations in terms of $\Xi$. Firstly, it describes the $h$-vector of this complex and secondly, it is related to the gallery distance on $\Xi$:
	
	\subsection{\texorpdfstring{$h$}{h}-vectors}
	The \emph{f-vector} of a non-empty finite complex $X$ of dimension $d-1$ is given by the tuple $f(X)=(f_{-1}, f_0, \ldots , f_{d-1})$, where $f_{-1}=1$ and for $i\geq 0$, $f_i$ denotes the number of $i$-faces of $X$. The \emph{h-vector} $h(X)=(h_0, \ldots , h_d)$ is defined from this by the linear relations
	\begin{align*}
		h_k \coloneqq \sum_{i=0}^k (-1)^{k-i} \binom{d-i}{k-i} f_{i-1}.
	\end{align*}
	Just like the $f$-vector, the $h$-vector encodes the number of faces of different dimensions of $X$. It has a particularly nice interpretation in the case where $X$ is partitionable (which is in particular the case for the shellable complex $\Xi$), see e.g. \cite[Proposition III.2.3]{Sta:CombinatoricsCommutativeAlgebra:}. 
	Hultman showed in \cite[Example 5.9, Theorem 5.10]{Hultman2007} that the $h$-polynomial of $\Xi$ equals the generating function of the statistic $t$, i.e. that one has
	\begin{align*}
		h(\Xi,x)= \sum_{i=0}^{d} h_i x^i = \sum_{w\in W} x^{\des(w)+\des(w^{-1})}.
	\end{align*}
	
	\subsection{Chamber complexes}
	Let $X$ be a \emph{pure} complex (i.e. all of its facets have the same dimension).
	Two facets of $X$ are called \emph{adjacent} if their intersection is a face of codimension $1$.
	The complex $X$ is called a \emph{chamber complex} if every pair of facets $\sigma,\tau\in X$ can be connected by a \emph{gallery}, i.e. a sequence of facets $\sigma=\nolinebreak\tau_0,\ldots, \tau_l=\tau$ such that for all $0\leq i \leq l$, the facets $\tau_i$ and $\tau_{i+1}$ are adjacent.
	In this setting, $l$ is called the \emph{length} of the gallery. For two facets $\sigma,\tau$ of a chamber complex $X$, the \emph{gallery distance} $d(\sigma, \tau)$ is defined as the minimal length of a gallery connecting $\sigma$ and $\tau$. Galleries of minimal length can be seen as the analogue of geodesics in the realm of chamber complexes.
	
	To see that $\Xi$ is a chamber complex, we first note that the facets of $\Xi$ are given by triples $(\emptyset, w, \emptyset)$, i.e. they are in one-to-one correspondence with the elements of $W$. Denote by $\sigma_w$ the facet corresponding to $w\in W$. Spelling out the definitions, it is easy to see that $\sigma_w$ and $\sigma_{w'}$ share a face of codimension $1$ if and only if $w'=ws$ or $w'=sw$ for some $s\in S$. Hence, the fact that $S$ generates $W$ implies that for any two facets of $\Xi$, there is a gallery connecting the two.
	
	In particular, for every $w\in W$, a gallery between the simplex $\sigma_e$ corresponding to the neutral element $e\in W$ and $\sigma_w$ corresponds to writing $w$ as a product of the elements in $S$. Furthermore, if $\sigma_e=\nolinebreak\sigma_{w_0},\ldots, \sigma_{w_l}=\sigma_w$ is a gallery of minimal length, we have
	\begin{align*}
		l_S(w_i)=i \text{ for all } 0\leq i \leq l,
	\end{align*}
	where $l_S(\cdot)$ denotes the word length with respect to $S$.
	One consequence of this is that the gallery distance $d(\sigma_e,\sigma_w)$ equals the word length $l_S(w)$. Furthermore, in such a gallery, there must be $s\in S$ such that $w_{l-1}=ws$ or $w_{l-1}=sw$ and $l_S(w_{l-1})=l_S(w)-1$.
	Noting that $s\in \on{Des}(w^{-1})$ if and only if 
	\begin{align*}
		l_S(w^{-1}s)=l_S((sw)^{-1})=l_S(sw)<l_S(w),
	\end{align*}
	we find the following, second interpretation of $t$ in terms of $\Xi$:
	\begin{observation}
		For any $w\in W$, the number of facets of $\Xi$ which are adjacent to $\sigma_w$ and lie on a gallery of minimal length between $\sigma_e$ and $\sigma_w$ is given by $t(w)=\des(w)+\des(w^{-1})$.
	\end{observation}
	In this sense, the statistic $t$ counts the number of geodesics starting at facets in $\Xi$.

	\section{Higher moments of \texorpdfstring{$T$}{T}}
	This section contains the higher moments of the random variables which were described in the proofs of Proposition~\ref{p:fourthmomentA} and Proposition~\ref{p:fourthmomentB}.
	
	Let $D_n=D_{W_n}$, $T_n=T_{W_n}$, let $D_n'$ be the random variable associated to the statistic
	\begin{align*}
		W_n &\to \mathbb{N}\\
		w &\mapsto \des(w^{-1})
	\end{align*}
	and define $U_n:=D_n-\EE (D_n) $ and $ U'_n:=D'_n-\EE (D'_n)$. 
	
	For the proofs of Proposition~\ref{p:fourthmomentA} and Proposition~\ref{p:fourthmomentB}, one needs to compute inductively $\EE(U_n^k {U_n'}^l)$ for all $0\leq k,l \leq 4$ where $W_n=\cox{A}_n$ and $W_n=\cox{B}_n$, respectively. Note that $\EE(U_n^k {U_n'}^l)=\EE(U_n^l {U_n'}^k)$.
	For the sake of completeness, we also list the mixed moments of $(D_n, D'_n)$, which can be computed similarly, although they are not needed to prove Proposition~\ref{p:fourthmomentA} and Proposition~\ref{p:fourthmomentB}.
	
	\subsection{Type \texorpdfstring{$\cox{A}$}{A}}
	\label{app:recursionmomentsA}
	For $W_n=\cox{A}_n$ we display the list of (joint) moments up to degree 4 in Table~\ref{tableA}. The result for $ \EE(U_n^4) $ corresponds to Lemma~\ref{l:fmA} and the result for $ \EE((T_n-\EE(T_n))^4) $ to Proposition~\ref{p:fourthmomentA}. The moments in boldface were already known before and can be found in \cite{KahleStump2018}.
	
	\begin{table}
		\renewcommand{\arraystretch}{1.5}
		\[\begin{array}[0.5cm]{c|c}
			& \EE(\cdot) \\ \hline
			\mathbf{U_n}	& 0\\
			\mathbf{U_n^2} 	& \frac{n+2}{12}\\
			\mathbf{U_nU_n'}& \frac{n}{2 (n+1)} \\
			U_n^3 			& 0\\
			U_n^2U'_n 		& 0 \\
			U_n^3U'_n		&\frac{n (n+2)}{8 (n+1)} \\
			U_n^4 			& \frac{1}{240} (n+2) (5 n+8)\\
			U_n^2{U'_n}^2	&\frac{1}{144} \left(n^2+4 n+76\right)-\frac{2 n+1}{3n(n+1)}\\
			\mathbf{(T_n-\EE(T_n))^2}&\frac{n+2}{6}+\frac{n}{n+1}\\
			(T_n-\EE(T_n))^3&0\\
			(T_n-\EE(T_n))^4& \frac{1}{60} \left(5 n^2+79 n+258\right)-\frac{5 n+2}{n(n+1)}\\
			\mathbf{D_n	} 	& \frac{n}{2}\\
			\mathbf{D_n^2} 	& \frac{n+2}{12}+\frac{n^2}{4}\\
			\mathbf{D_nD_n'}& \frac{n^2}{4}+\frac{n}{2n+2} \\
			D_n^3 			& \frac{n(n^2+n+2)}{8}\\
			D_n^2D'_n 		& \frac{1}{24}(3 n^3+n^2+14 n-12 ) +\frac{1}{2 (n+1)} \\
			D_n^3D'_n		& \frac{1}{16}(n^4-4 n^3+15 n^2-36 n+56 ) -\frac{4}{ (n+1)} \\
			D_n^4 			& \frac{1}{240}(15 n^4+30 n^3+65 n^2+18 n+16)\\
			D_n^2{D'_n}^2	&\frac{1}{144}(9 n^4+6 n^3+85 n^2-68 n+148) -\frac{7 n+2}{6 n(n+1)}\\
			\mathbf{T_n^2}	& n^2+\frac{n+2}{6}+\frac{n}{n+1}\\
			T_n^3			& n^3+\frac{n^2}{2}+4 n-3+\frac{3}{n+1}\\
			T_n^4			&n^4+n^3+\frac{97 n^2}{12}-\frac{281 n}{60}+\frac{103}{10}-\frac{11 n+2}{n(n+1)}
		\end{array}\]
	\caption{List of moments for type $ \cox{A}_n $.}\label{tableA}
	\end{table}
	
	\subsection{Type \texorpdfstring{$\cox{B}$}{B}}
	\label{app:recursionmomentsB}
	For $W_n=\cox{B}_n$ we display the list of (joint) moments up to degree 4 in Table~\ref{tableB}. The result for $ \EE(U_n^4) $ corresponds to Lemma~\ref{l:fmB} and the result for $ \EE((T_n-\EE(T_n))^4) $ to Proposition~\ref{p:fourthmomentB}. The moments in boldface  were already known before and can be found in \cite{KahleStump2018}.
	
	\begin{table}
		\renewcommand{\arraystretch}{1.5}
		\[\begin{array}[0.5cm]{c|c}
			& \EE(\cdot) \\ \hline
			\mathbf{U_n}	& 0\\
			\mathbf{U_n^2} 	& \frac{n+1}{12}\\
			\mathbf{U_nU_n'}& \frac{1}{4} \\
			U_n^3 			& 0\\
			U_n^2U'_n 		& 0 \\
			U_n^3U'_n		& \frac{n+1}{16} \\
			U_n^4 			& \frac{1}{240} (n+1) (5 n+3)\\
			U_n^2{U'_n}^2	&\frac{1}{144} \left(n^2+2 n+19\right)+\frac{2 n-1}{24 n(n-1)}\\
			\mathbf{(T_n-\EE(T_n))^2}&\frac{n+4}{6}\\
			(T_n-\EE(T_n))^3&0\\
			(T_n-\EE(T_n))^4& \frac{1}{60} \left(5 n^2+39 n+79\right)+\frac{2 n-1}{4n(n-1)}\\
			\mathbf{D_n	} 	& \frac{n}{2}\\
			\mathbf{D_n^2} 	& \frac{n+1}{12}+\frac{n^2}{4}\\
			\mathbf{D_nD_n'}& \frac{n^2+1}{4} \\
			D_n^3 			& \frac{n(n^2+n+1)}{8}\\
			D_n^2D'_n 		& \frac{1}{24}  n (7 + n + 3 n^2)\\
			D_n^3D'_n		&\frac{1}{16} (1 + n + 4 n^2 + n^3 + n^4) \\
			D_n^4 			& \frac{1}{240} (15 n^4+30 n^3+35 n^2+8 n+3)\\
			D_n^2{D'_n}^2	&\frac{1}{144}(9 n^4+6 n^3+43 n^2+2 n+19) +\frac{2 n-1}{24 n(n-1)}\\
			\mathbf{T_n^2}& n^2+\frac{n+4}{6}\\
			T_n^3			& n(n^2+\frac{n}{2}+2)\\
			T_n^4			& n^4+n^3+\frac{49 n^2}{12}+\frac{13 n}{20}+\frac{79}{60} +\frac{2 n-1}{4 n(n-1)}
		\end{array}\]
	\caption{List of moments for type $ \cox{B}_n $.}\label{tableB}
	\end{table}
	\clearpage
\end{appendix}

\bibliographystyle{plain}
\bibliography{bibliography}

\end{document}